\documentclass[12pt,a4paper]{amsart}

\usepackage{amsmath}
\usepackage{amsfonts}
\usepackage{amssymb}
\usepackage{graphicx}

\newtheorem{theorem}{Theorem}[]
\newtheorem{proposition}{Proposition}[section]
\newtheorem{corollary}[proposition]{Corollary}
\newtheorem{lemma}[proposition]{Lemma}
\theoremstyle{definition}

\newtheorem{remark}[proposition]{Remark}
\newtheorem*{acknowledgements}{Acknowledgements}

\newcommand{\R}{\mathbb{R}} 
\newcommand{\C}{\mathbb{C}} 
\newcommand{\dd}{\mathrm{d}} 
\newcommand{\y}{\mathbf{y}} 
\newcommand{\x}{\mathbf{x}} 
\newcommand{\T}{\mathbf{t}} 

\newcommand{\norm}[3]{\left\|#1\right\|^{#2}_{#3}} 
\newcommand{\Duality}[2]{\left\langle #1 \Big| #2 \right\rangle} 
\newcommand{\inner}[2]{\left\langle #1, #2 \right\rangle} 

\newcommand{\mR}{\mathbb{R}} 
\newcommand{\normi}[1]{\lVert #1 \rVert}
\newcommand{\abs}[1]{\lvert #1 \rvert}

\DeclareMathOperator{\supp}{supp}
\DeclareMathOperator{\dist}{dist}

\DeclareMathOperator{\inte}{int}
\DeclareMathOperator{\spec}{spec}

\author[]{Pedro Caro \and Mikko Salo}
\title[]{Stability of the Calder\'on problem in admissible geometries}
\date{}
\address{Department of Mathematics and Statistics, Helsingin yliopisto / Helsingfors universitet / University of Helsinki, Finland}
\email{pedro.caro@helsinki.fi}
\address{Department of Mathematics and Statistics, University of Jyv\"askyl\"a, Finland}
\email{mikko.j.salo@jyu.fi}

\begin{document}

\begin{abstract}
In this paper we prove \textit{log log} type stability estimates for inverse boundary value problems on admissible Riemannian manifolds of dimension $n \geq 3$. The stability estimates correspond to the uniqueness results in \cite{DSFKSaU}. These inverse problems arise naturally when studying the anisotropic Calder\'on problem.
\end{abstract}

\maketitle


\section{Introduction}

\noindent {\bf Background.} In the inverse conductivity problem of Calder\'on \cite{C}, the objective is to determine the electrical properties of a medium from voltage and current measurements on its boundary. Suppose that the medium is modelled by a bounded open set $\Omega \subset \mR^n$ with Lipschitz boundary, and let $\gamma = (\gamma^{jk}) \in L^{\infty}(\Omega, \mR^{n \times n})$ be a positive definite symmetric matrix function describing the electrical conductivity. Then for any boundary voltage $f$, the voltage potential $u$ in the medium satisfies the conductivity equation, 
$$
\mathrm{div}(\gamma \nabla u) = 0 \text{ in } \Omega, \quad u|_{\partial \Omega} = f.
$$
The boundary measurements are encoded by the Dirichlet-to-Neumann map (DN map for short) 
$$
\Lambda_{\gamma}: f \mapsto \gamma \nabla u \cdot \nu|_{\partial \Omega}
$$
where $\nu$ is the unit outer normal of $\partial \Omega$. Using a suitable weak definition, the DN map becomes a bounded linear operator 
$$
\Lambda_{\gamma}: H^{1/2}(\partial \Omega) \to H^{-1/2}(\partial \Omega)
$$
where $H^s(\partial \Omega)$ is the $L^2$ based Sobolev space on $\partial \Omega$. The inverse problem is to determine properties of the unknown conductivity function $\gamma$ from the knowledge of the map $\Lambda_{\gamma}$.

Assume now that the conductivity is \emph{isotropic}, that is, 
$$
\gamma^{jk}(x) = \gamma(x) \delta^{jk}
$$
where $\gamma \in L^{\infty}(\Omega)$ is a positive function. One can ask the following basic questions for the Calder\'on problem with isotropic conductivities:
\begin{enumerate}
\item[1.] 
{\bf Uniqueness:} does $\Lambda_{\gamma_1} = \Lambda_{\gamma_2}$ imply $\gamma_1 = \gamma_2$?
\item[2.] 
{\bf Reconstruction:} find an algorithm for computing $\gamma$ from $\Lambda_{\gamma}$.
\item[3.] 
{\bf Stability:} if $\Lambda_{\gamma_1}$ and $\Lambda_{\gamma_2}$ are close, are also $\gamma_1$ and $\gamma_2$ close?
\end{enumerate}
Both the theoretical and applied aspects of the Calder\'on problem have been under intense study, and we refer to the survey \cite{U_IP} for more information. In particular, there are several uniqueness results \cite{AP}, \cite{HT}, \cite{N_2D}, \cite{SU} and reconstruction procedures \cite{regulDbar}, \cite{N_reconstruction}. In this paper we are interested in stability results, and we proceed to describe these in more detail.

The fundamental stability result due to Alessandrini \cite{A} states that if the coefficients $\gamma_1$ and $\gamma_2$ satisfy a priori bounds in $H^{s+2}(\Omega)$ for $s > n/2$ where $n \geq 3$, then 
$$
\normi{\gamma_1-\gamma_2}_{L^{\infty}(\Omega)} \leq \omega(\normi{\Lambda_{\gamma_1}-\Lambda_{\gamma_2}}_{H^{1/2} \to H^{-1/2}})
$$
where $\omega$ is a modulus of continuity satisfying 
\begin{equation*}
\omega(t) \leq C \abs{\log\,t}^{-\sigma}, \quad 0 < t < 1/e
\end{equation*}
with $C$ depending on the a priori bounds.
This \emph{log} type stability for the Calder\'on problem (as opposed to H\"older or Lipschitz stability) and the required a priori bounds express the fact that this inverse problem is highly ill-posed. It has been shown that logarithmic stability is optimal for the Calder\'on problem \cite{Mandache}, although if one has a priori information then one may have better stability properties \cite{AlessandriniVessella}. There are several related stability results in the literature as \cite{CaGR} and \cite{CFR}. We refer to the survey \cite{AlessandriniSurvey} for further references. We also mention that in practice, the measured DN map in presence of noise may not coincide with a DN map for any conductivity, and to rectify this the stability analysis has been combined with a regularization procedure in \cite{regulDbar} for $n=2$.

\vspace{8pt}

\noindent {\bf Anisotropic Calder\'on problem.} In this paper we study stability for the Calder\'on problem with \emph{anisotropic} conductivities, where $\gamma(x)$ is a matrix function which may not be a scalar multiple of the identity matrix. It is well known that the anisotropic Calder\'on problem has a simple obstruction to uniqueness: given any anisotropic conductivity $\gamma$ defined in $\Omega$ with smooth boundary and any diffeomorphism $F : \overline{\Omega} \longrightarrow \overline{\Omega}$ satisfying $F|_{\partial \Omega} = \mathrm{Id}$, one has 
$$
\Lambda_\gamma = \Lambda_{F_\ast \gamma}.
$$ Here $F_\ast \gamma$ is the pushforward conductivity 
\[F_\ast \gamma (x) = \left. \frac{DF \, \gamma \, DF^t}{\mathrm{det}\,DF} \right|_{F^{-1}(x)}, \]
where $DF $ denotes the matrix given by $ ( \partial_{x_j} F^k )$ and $DF^t$ is its transpose. It is known that when $n=2$, the DN map $\Lambda_{\gamma}$ determines $\gamma$ up to such a diffeomorphism \cite{N_2D}, \cite{ALP}, but for $n \geq 3$ this is only known for real-analytic conductivity matrices \cite{LU}. A simplification of the anisotropic Calder\'on problem which avoids this obstruction consists in assuming that $\gamma^{jk} = \sigma \gamma_0^ {jk}$ with the matrix $(\gamma_0^ {jk})$ being known and trying to recover  the scalar function $\sigma$ from $\Lambda_\gamma$. Note that if $(\gamma_0^ {jk})$ is the identity matrix, this is just the Calder\'on problem for isotropic conductivities.

As was pointed out in \cite{LU}, whenever the conductivity is smooth and $n \geq 3$ the anisotropic Calder\'on problem is of geometrical nature and it can be formulated in Riemannian manifolds as follows. Let $(M, g)$ be an oriented compact Riemannian $n$-dimensional manifold with boundary $\partial M$ and $n = \dim(M) \geq 3$. The Laplace-Beltrami operator associated to the metric $g = (g_{jk})$ and applied to a smooth function $u$ can be written in local coordinates as
\[\Delta_g u = |g|^{-1/2} \partial_{x_j} (|g|^{1/2} g^{jk} \partial_{x_k} u)\]
where $(g^{jk})$ is the inverse matrix of $(g_{jk})$ and $|g|$ is the determinant of $(g_{jk})$. Here we are using Einstein's summation convention: repeated indices in upper and lower position are summed. Consider $u \in H^1(M)$ solving $ -\Delta_gu = 0$ in $M^\mathrm{int}$ such that $u|_{\partial M} = f$ and define the DN map $ \Lambda_g : H^{1/2} (\partial M) \longrightarrow H^{-1/2} (\partial M)$ by
\begin{equation*}
\Duality{ \Lambda_g f }{ \phi|_{\partial M} } = \int_M \langle d u, d \phi \rangle_g \, \dd V_g
\end{equation*}
for any $f \in H^{1/2}(\partial M)$ and any $ \phi \in H^1(M) $. Here  $ \langle \centerdot \vert \centerdot \rangle $ denotes the duality between $ H^{1/2} (\partial M) $ and $ H^{-1/2} (\partial M) $. If $ f $ is smooth enough one can check that $ \Lambda_g f = g(\nu, \nabla u)|_{\partial M} = du(\nu)|_{\partial M} = \nu(u)|_{\partial M} $ where $ \nu $ represents the unit outer normal to $ \partial M $. Now, the Calder\'on problem on manifolds consists in recovering $g$ up to a boundary fixing diffeomorphism from $\Lambda_g$. Once again, it makes sense to consider the simplification where $g$ belongs to a fixed conformal class defined by some metric $g'$ and one tries to recover the unknown conformal factor from $\Lambda_g$. Also here one can consider different aspects such as uniqueness, reconstruction and stability. Here, we will study the question of stability in the conformal class defined by an \textit{admissible} metric $g'$.

\vspace{8pt}

\noindent{\bf Inverse problem for Schr\"odinger equation.} It turns out that the Calder\'on problem in a fixed conformal class can be reduced to the inverse boundary value problem (IBVP) of determining the electric potential of a Schr\"odinger operator on a compact Riemannian manifold from boundary measurements of all its solutions. In order to set up this problem, we consider  an oriented compact Riemannian $n$-dimensional manifold $(M, g)$, with boundary $\partial M$ and dimension $n \geq 2$, and an electric potential $q \in L^\infty(M)$. We define the Cauchy data set of $H^1$ solutions to the Schr\"odinger operator $ -\Delta_g + q$ as the set, denoted by $C_q$, of pairs $(f,w) \in H^{1/2} (\partial M) \times H^{-1/2} (\partial M)$ for which there exists 
$u \in H^1(M)$ solving $(-\Delta_g + q)u = 0$ in $M^\mathrm{int}$ such that $u|_{\partial M} = f$ and
\begin{equation}
\Duality{ w }{ \phi|_{\partial M} } = \int_M (\langle d u, d \phi \rangle_g + q u \phi) \, \dd V_g \label{id:normalDERIVATIVE}
\end{equation}
for any $ \phi \in H^1(M) $. Here  $ \langle \centerdot \vert \centerdot \rangle $ denotes the duality between $ H^{1/2} (\partial M) $ and $ H^{-1/2} (\partial M) $. For other notations used here and throughout the text see the paragraph Notation at the end of this section. Again, if $ f $ is smooth enough one can check that $ w = g(\nu, \nabla u)|_{\partial M} = du(\nu)|_{\partial M} = \nu(u)|_{\partial M} $ where $ \nu $ represents the unit outer normal to $ \partial M $. Thus, the IBVP under consideration consists in determining the electric potential $q$ from the Cauchy data set $C_q$. Associated to this problem there are several relevant questions, namely, uniqueness, reconstruction and stability. In this paper, we will consider the question of stability in the case where $g$ is in the conformal class of an \textit{admissible} metric $g'$ (that is $g = cg'$ with $c$ denoting the conformal factor) and $n \geq 3$. 

In order to establish the relation between the IBVP for Schr\"odinger operator and the anisotropic Calder\'on problem, it is enough to note that $u \in H^1(M)$ is solution of $-\Delta_g u = 0$ in $M$ with $g = cg'$ if and only if $v = c^{\frac{n - 2}{4}} u \in H^1(M)$ is a solution of the Schr\"odinger equation $-\Delta_{g'}v + qv = 0$ with $q = c^{-\frac{n - 2}{4}} \Delta_{g'} c^{\frac{n - 2}{4}} $. Thus, knowing the conformal factor $c$ on $\partial M$ we can relate $\Lambda_g$ with $C_q$ for the matrix $g'$. This sort of relation will be used for studying the questions already mentioned (see Section \ref{sec:boundaryINTERIOR} below).

\vspace{8pt}

\noindent{\bf Main results.} We next describe the main results in this paper. Let $ (M_0, g_0) $ be a simple\footnote{A compact manifold $(M, g)$ with boundary is called simple if, for any point $p \in M$, the exponential map $\exp_p$ is a diffeomorphism from its maximal domain in $ T_p M $ onto $M$ and the boundary $\partial M$ is strictly convex.} Riemannian oriented smooth compact $ (n - 1) $-dimensional manifold (for $ n \geq 3 $) with boundary $ \partial M_0 $. Assume $ (M, g) $ to be a Riemannian oriented smooth compact manifold with boundary such that there exist a smooth $ n $-dimensional embedded sub-manifold of $ \R \times M_0^\mathrm{int} $, namely $ M' $, an orientation preserving diffeomorphism $ F : M \longrightarrow M' $ --whose inverse will be denoted by $ G $-- and a positive smooth function $ c : M \longrightarrow (0 , + \infty) $ satisfying
\[ g = c F^\ast g', \]	
where $ g' = (e_\R \oplus g_0)|_\Omega $ and $ e_\R $ stands for the euclidean metric in $ \R $. A manifold $(M,g)$ as above will be called, throughout the paper, \emph{admissible}.

We now state the stability estimates  for the IBVP of recovering the electric potential $q$ from the Cauchy data set $C_q$. First we introduce the notion of proximity for Cauchy data sets that will be used to state the stability estimates. Let $q_1$ and $q_2$ belong to $L^\infty(M)$ and consider the Cauchy data sets $C_{q_1}$ and $C_{q_2}$ as above. Define the pseudo-metric distance
\[\dist (C_{q_1}, C_{q_2}) = \max_{j, k \in \{ 1, 2 \}} \sup_{\substack{(f_j, w_j) \in C_{q_j} \\ \norm{f_j}{}{H^{1/2}(\partial M)} = 1}} I ((f_j, w_j); C_{q_k})\]
where
\[ I ((f_j, w_j); C_{q_k}) = \inf_{(f_k, w_k) \in C_{q_k}} \left[ \norm{f_j - f_k}{}{H^{1/2}(\partial M)} + \norm{w_j - w_k}{}{H^{-1/2}(\partial M)} \right]. \]

\begin{theorem}\label{th:sch} \sl Consider a constant $K \geq 1$ and let $(M, g)$ be admissible. There exists a constant $C \geq 1$ depending on $M$ and $g$ such that
\begin{align*}
\norm{q_1 - q_2}{}{L^2(\mathbb{R}; H^{-3} (M_0))} \lesssim & \bigg| \log \big(\dist (C_{q_1}, C_{q_2}) \\
& + |\log \dist (C_{q_1}, C_{q_2})|^{-1} \big) \bigg|^{-\lambda/4}
\end{align*}
whenever $q_1, q_2 \in L^\infty (M) \cap H^\lambda (M)$ with $\lambda \in (0, 1/2)$ satisfy $\norm{q_j}{}{L^\infty(M)} + \norm{q_j}{}{H^\lambda (M)} \leq K$ and $\dist (C_{q_1}, C_{q_2}) \leq e^{-CK}$. Here the implicit constants only depend on $M, g, n, K$ and $\lambda$.
\end{theorem}
Note that we are making an abuse of notation writing $q_j$ instead of the extension by zero of $G^\ast q_j$ out of $M'$.
\begin{remark} \sl
Assuming a priori bounds for stronger norms of $q_j$, we can replace the norm on the left hand side of the stability estimate by stronger norms only losing some power on the right hand side. This can be achieved using appropriate interpolation arguments (see for example \cite{A}).
\end{remark}

We next state the stability estimates for the Calder\'on problem in a fixed conformal class of an admissible metric. First recall the operator norm that we will use to quantify the proximity between the Dirchlet-to-Neumann maps:
\[ \| \Lambda_g \|_\ast = \sup_{f \in H^{1/2}(\partial M) \setminus \{ 0 \}} \frac{\| \Lambda_g f \|_{H^{-1/2}(\partial M)}}{\| f \|_{H^{1/2}(\partial M)}}. \]

\begin{theorem}\label{th:cal} \sl Consider a constant $K \geq 1$ and an admissible manifold $(M,g)$«. Let $g_1$ and $g_2$ be two metrics on $M$ satisfying $g_j = c_j F^\ast g'$ with $F^\ast$ and $g'$ as above. If $c_1$ and $c_2$ are smooth and $\norm{c_j^{-1}}{}{L^\infty (M)} + \norm{c_j}{}{C^3 (M)} \leq K$, there exists a constant $C \geq 1$ depending on $M'$, $g'$ and $n$ such that
\begin{equation*}
\norm{c_1 - c_2}{}{L^\infty (M)} \lesssim \bigg| \log \big( \norm{\Lambda_{g_1} - \Lambda_{g_2}}{}{\ast} + |\log \norm{\Lambda_{g_1} - \Lambda_{g_2}}{}{\ast}|^{-1} \big) \bigg|^\theta
\end{equation*}
whenever $\norm{\Lambda_{g_1} - \Lambda_{g_2}}{}{\ast} \leq e^{-CK}$. Here $\theta$ is a small positive constant which depends on $n$. The implicit constants only depend on $M', g', F, n$ and $K$.

\end{theorem}

In order to prove these theorems we will follow the standard argument based on complex geometrical optics solutions (CGOs for short). The first step is to use an integral identity that relates the unknowns in the interior with the boundary measurements. The second step is to extract information on the unknowns by using special solutions for the equation, namely, CGOs. In our case the information is described by a mixed Fourier transform/attenuated geodesic ray transform. More precisely, we are able to prove an estimate controlling a rather weak norm of the attenuated geodesic ray transform, with attenuation $\sigma$, of $\hat{q}(\sigma)$ --the Fourier transform of the unknown in the Euclidean direction at frequency $\sigma$. This estimate can be rephrased in terms of the normal operator for the ray transform, which is an elliptic operator of order $-1$. Thanks to the ellipticity of the normal operator, we manage to obtain control of $\hat{q}(\sigma)$ for a small set of low frequencies $\sigma$. By using analytic continuation, we enlarge the set of low frequencies and as a consequence we prove an inequality bounding a weak norm of the unknown. Finally, standard interpolation arguments yield the stability stated in Theorem \ref{th:sch} and Theorem \ref{th:cal}.

As we mentioned above, the sharp stability estimate of the isotropic Calder\'on problem is of \textit{log} type. Here we only prove \textit{log log} stability estimates. The extra \textit{log} in our results comes up because of the analytic continuation argument that enlarges the set of controlled frequencies. The small size of this set is due to the fact that we only apply injectivity of the attenuated geodesic ray transform for small attenuations. However, injectivity of the attenuated geodesic ray transform for larger attenuation would not imply \textit{log} stability following our approach. One can check that the implicit constant in Lemma \ref{lem:normalOPstability} (below) grows at least exponentially as $\delta_0$ increase. This together with the exponential factor in the estimate \eqref{es:stabNORMAL} would produce a second \textit{log} in the final stability estimate. Despite this second \textit{log} for the stability of the whole problem, we could gain better control from knowing the injectivity of the attenuated geodesic ray transform for larger attenuation, namely, we would be able to prove \textit{log} type stability for the low frequencies of the Fourier transform of the unknown in the Euclidean direction. This stability would become exponentially bad with the size of the low-frequency set. Injectivity of the attenuated ray transform on simple surfaces for any attenuation has been proven in \cite{SaU}. We mention that also in stability results for the Calder\'on problem with partial data, both \textit{log} estimates (\cite{Caro_stability}, \cite{HW2}) and \textit{log log} estimates (\cite{CDSFR}, \cite{CaDSFR}, \cite{HW}) appear.

The arguments we use to prove Theorem \ref{th:sch} and Theorem \ref{th:cal} are a quantification of the arguments in \cite{DSFKSaU} that prove uniqueness results for the above inverse problems. The approach in \cite{DSFKSaU} has been recently followed in \cite{CaDSFR} to prove \textit{log log} stability estimates for the Calder\'on problem with partial data. The quantification argument there is slightly different to ours. In \cite{CaDSFR}, the authors do not use explicitly the ellipticity of the normal operator, they prove a direct estimate for the attenuated geodesic ray transform.

The outline of this paper is as follows. In Section \ref{sec:boundaryINTERIOR} we provide the integral estimates that will be used later as starting points to prove the stability estimates given in the theorems stated above. In Section \ref{sec:CGOs} we review the construction of the CGOs given in \cite{DSFKSaU}. Finally, in Section \ref{sec:stability} we prove the stability estimates.

\vspace{8pt}
\noindent {\bf Notation. }\label{subsec:notations}
Throughout this paper:
\begin{itemize}
\item $ M^\mathrm{int} = M \setminus  \partial M $
\item $\Delta_g$, $\inner{\centerdot}{\centerdot}_g$ and $\dd V_g$ denote respectively the Laplace-Beltrami operator, the inner product for differential forms and the volume form associated to the Riemannian  metric $g$.
\item A Riemannian metric $g$ is denoted in local coordinates by the matrix $(g_{jk})$. Moreover, the inverse and the determinant of this matrix are denoted by $(g^{jk})$ and $|g|$.
\item If $F$ is a smooth map, $ F_\ast $ and $ F^\ast $ denote the push-forward and pull-back respectively.
\end{itemize}

\section{From the boundary to the interior} \label{sec:boundaryINTERIOR}
In this section we prove two integral identities, one for the IBVP for the Schr\"odinger operator and one for the generalized Calder\'on problem. These identities relate the unknowns in the interior with the corresponding boundary data. The notation is as in the introduction.

\begin{proposition} \label{prop:boundaryESq1q2} \sl Let $ q_1 $ and $ q_2 $ belong to $ L^\infty (M) $ and let $ C_{q_1} $ and $ C_{q_2} $ denote the Cauchy data sets for $H^1(M)$ solutions of the operators $-\Delta_g + q_1$ and $-\Delta_g + q_2$, with $g = c F^\ast g'$. Then for any $v_j \in H^1(M')$ with $j \in \{ 1, 2 \}$ solving the equation
\[-\Delta_{g'} v_j + ( c^{-\frac{n - 2}{4}} \Delta_{g'} c^\frac{n - 2}{4}  + c q_j) v_j = 0 \]
in $M'_\mathrm{int}$, we have
\begin{align*}
\bigg| \int_{M'} c &(q_1 - q_2) v_1 v_2 \, \dd V_{g'} \bigg| \\
& \lesssim \dist (C_{q_1}, C_{q_2}) Q \norm{v_1}{}{H^1(M')} \norm{v_2}{}{H^1(M')}
\end{align*}
where $ Q = \max\{ 1 + \norm{q_j}{}{L^\infty(M)}: j = 1, 2 \} $. Here we are making an abuse of notation which consists in writing $q_j$ and $c$ instead of $G^\ast q_j$ and $G^\ast c$. The implicit constant in the inequality depends on $n$, $c$, $M$, $M'$, $ g' $ and $F$.
\end{proposition}
\begin{proof}
Let $u_j$ with $j\in \{1,2\}$ be defined by $u_j = c^{-\frac{n - 2}{4}} F^\ast v_j$. Then $u_j $ belongs to $ H^1(M)$ and it is a solution to $(-\Delta_g + q_j) u_j = 0$ in $M$. Let us define $\nu(u_j)$ in the weak form as in \eqref{id:normalDERIVATIVE}, then
\begin{align*}
\Big\langle \nu(u_j) & \Big| u_k \Big\rangle = \int_M \langle d u_j, d u_k \rangle_g + q_j u_j u_k \, \dd V_g \\
&= \int_{M'} \langle d v_j, d v_k \rangle_{g'} - \langle dc^\frac{n- 2}{4}, d(c^{-\frac{n- 2}{4}} v_j v_k) \rangle_{g'} + cq_j v_j v_k \, \dd V_{g'}
\end{align*}
with $j, k \in \{ 1, 2 \}$. We are making an abuse of notation in the left hand side writing $u_k$ instead of $u_k|_{\partial M}$. Thus, it is immediate to get
\[ \Big\langle \nu(u_1)  \Big| u_2 \Big\rangle - \Big\langle \nu(u_2) \Big| u_1 \Big\rangle = \int_{M'} c(q_1 - q_2) v_1 v_2 \, \dd V_{g'}. \]
Since,
\[\Big\langle w_2 \Big| u_2 \Big\rangle - \Big\langle \nu(u_2) \Big| f_2 \Big\rangle = 0 \]
for any $(f_2, w_2) \in C_{q_2}$, we have that
\[ \Big\langle \nu(u_1) - w_2  \Big| u_2 \Big\rangle - \Big\langle \nu(u_2) \Big| u_1 - f_2 \Big\rangle = \int_{M'} c(q_1 - q_2) v_1 v_2 \, \dd V_{g'} \]
for any $(f_2, w_2) \in C_{q_2}$. By the definition of $\dist(C_{q_1}, C_{q_2})$, the estimate
\[\norm{u_j}{}{H^{1/2}(\partial M)} + \norm{\nu(u_j)}{}{H^{-1/2}(\partial M)} \lesssim (1 + \norm{q_j}{}{L^\infty(M)}) \norm{u_j}{}{H^1(M)},\]
and the definition of $u_j$ we get the statement of the proposition.
\end{proof}

\begin{proposition} \label{prop:int_forCAL} \sl Let $ g_1 $ and $ g_2 $ be two metrics on $ M $ satisfying $g_j = c_j F^\ast g'$. Let $ \Lambda_{g_1} $ and $ \Lambda_{g_2} $ denote their corresponding DN maps. Then, for any $ v_j \in H^1(M') $ with $ j \in \{ 1, 2 \} $ solving
\[-\Delta_{g'} v_j + c^{-\frac{n - 2}{4}}_j \Delta_{g'} c^\frac{n - 2}{4}_j v_j = 0 \]
in $M'_\mathrm{int}$, we have
\begin{align*}
\bigg| \int_{M'} &(c_1^{-\frac{n- 2}{4}} \Delta_{g'} c_1^\frac{n- 2}{4} - c_2^{-\frac{n- 2}{4}} \Delta_{g'} c_2^\frac{n- 2}{4}) v_1 v_2 \, \dd V_{g'} \bigg| \\
\lesssim & C \Big( \norm{\Lambda_{g_1} - \Lambda_{g_2}}{}{\ast} + \norm{c_2^{-\frac{n - 2}{4}} - c_1^{-\frac{n - 2}{4}}}{}{C^1(\partial M)} \\
& + \norm{\nu( c_2^\frac{n- 2}{4} ) - \nu( c_1^\frac{n- 2}{4} )}{}{L^\infty (\partial M)} \Big) \norm{v_1}{}{H^1(M')} \norm{v_2}{}{H^1(M')},
\end{align*}
where $ C = \max\{ 1 + \norm{c_j}{1/2}{L^\infty(M')} + \norm{d (\mathrm{log}\, c_j)}{}{L^\infty(M')}: j = 1, 2 \} $. Here we are making an abuse of notation which consists in writing $c_j$ instead of $G^\ast c_j$. The implicit constant in the inequality depends on $n$, $M$, $M'$, $ g' $ and $F$.
\end{proposition}
\begin{proof}
Let $u_j$ with $j\in \{1,2\}$ be defined by $u_j = c_j^{-\frac{n - 2}{4}} F^\ast v_j$. Then $u_j $ belongs to $ H^1(M)$, it is solution to $-\Delta_{g_j} u_j = 0$ in $M$ and 
\begin{align*}
\Big\langle \Lambda_{g_j}(u_j) & \Big| u_k \Big\rangle = \Big\langle \Lambda_{g_j}(u_j)  \Big| c_j^{-\frac{n - 2}{4}} v_k \Big\rangle + \Big\langle \Lambda_{g_j}(u_j) \Big| (c_k^{-\frac{n - 2}{4}} - c_j^{-\frac{n - 2}{4}})v_k \Big\rangle \\
&= \int_M \langle d u_j, d (c_j^{-\frac{n - 2}{4}} v_k) \rangle_{g_j} \, \dd V_{g_j} + \Big\langle \Lambda_{g_j}(u_j) \Big| (c_k^{-\frac{n - 2}{4}} - c_j^{-\frac{n - 2}{4}})v_k \Big\rangle
\end{align*}
with $j, k \in \{ 1, 2 \}$. Here we are making an abuse of notation writing $u_j$, $u_k$ and $ v_k $ instead of $u_j|_{\partial M}$, $u_k|_{\partial M}$ and $F^\ast v_k|_{\partial M}$ (on the boundary) or $ F^\ast v_k $ (in the interior). On the other hand,
\begin{align*}
\int_M \langle d u_j, d &(c_j^{-\frac{n - 2}{4}} v_k) \rangle_{g_j} \, \dd V_{g_j} \\
=& \int_{M'} \langle d v_j, d v_k \rangle_{g'} - \langle dc_j^\frac{n- 2}{4}, d(c_j^{-\frac{n- 2}{4}} v_j v_k) \rangle_{g'} \, \dd V_{g'} \\
=& \int_{M'} \langle d v_j, d v_k \rangle_{g'} + c_j^{-\frac{n- 2}{4}} \Delta_{g'} c_j^\frac{n- 2}{4} v_j v_k \, \dd V_{g'}\\
& - \int_{\partial M'} c_j^{-\frac{n- 2}{4}} \nu( c_j^\frac{n- 2}{4} ) v_j v_k \, \dd A_{g'}
\end{align*}
where $\dd A_{g'} $ is the contraction of $\dd V_{g'}$ with $\nu$. Again, we are making an abuse of notation consisting in writing $ \nu $ instead of $ F_\ast \nu $. Thus, it is immediate to get
\begin{align*}
\Big\langle \Lambda_{g_1}(&u_1) \Big| u_2 \Big\rangle - \Big\langle \Lambda_{g_2}(u_2) \Big| u_1 \Big\rangle \\
=& \Big\langle \Lambda_{g_1}(u_1) \Big| (c_2^{-\frac{n - 2}{4}} - c_1^{-\frac{n - 2}{4}})v_2 \Big\rangle
- \Big\langle \Lambda_{g_2}(u_2) \Big| (c_1^{-\frac{n - 2}{4}} - c_2^{-\frac{n - 2}{4}})v_1 \Big\rangle \\
& + \int_{M'} (c_1^{-\frac{n- 2}{4}} \Delta_{g'} c_1^\frac{n- 2}{4} - c_2^{-\frac{n- 2}{4}} \Delta_{g'} c_2^\frac{n- 2}{4}) v_1 v_2 \, \dd V_{g'} \\
& + \int_{\partial M'} ( c_2^{-\frac{n- 2}{4}} \nu( c_2^\frac{n- 2}{4} ) - c_1^{-\frac{n- 2}{4}} \nu( c_1^\frac{n- 2}{4} ) ) v_1 v_2 \, \dd A_{g'}.
\end{align*}
Since,
\[\Big\langle \Lambda_{g_2}(u_2) \Big| u_1 \Big\rangle = \Big\langle \Lambda_{g_2}(u_1) \Big| u_2 \Big\rangle\]
we have
\begin{align*}
\bigg| \int_{M'} &(c_1^{-\frac{n- 2}{4}} \Delta_{g'} c_1^\frac{n- 2}{4} - c_2^{-\frac{n- 2}{4}} \Delta_{g'} c_2^\frac{n- 2}{4}) v_1 v_2 \, \dd V_{g'} \bigg| \\
\leq & \norm{\Lambda_{g_1} - \Lambda_{g_2}}{}{\ast} \norm{u_1}{}{H^1(M)} \norm{u_2}{}{H^1(M)} \\
+&  \bigg| \Big\langle \Lambda_{g_1}(u_1) \Big| (c_2^{-\frac{n - 2}{4}} - c_1^{-\frac{n - 2}{4}})v_2 \Big\rangle \bigg|
+ \bigg| \Big\langle \Lambda_{g_2}(u_2) \Big| (c_1^{-\frac{n - 2}{4}} - c_2^{-\frac{n - 2}{4}})v_1 \Big\rangle \bigg| \\
+&  \norm{c_2^{-\frac{n- 2}{4}} \nu( c_2^\frac{n- 2}{4} ) - c_1^{-\frac{n- 2}{4}} \nu( c_1^\frac{n- 2}{4} )}{}{L^\infty (\partial M)} \norm{v_1}{}{H^1(M')} \norm{v_2}{}{H^1(M')},
\end{align*}
where $ \norm{\centerdot}{}{\ast} $ denotes the norm of the bounded operators from $ H^{1/2} (\partial M) $ to $ H^{-1/2} (\partial M) $. On one hand,
\[ \norm{u_j}{}{H^1(M)} \lesssim \left( 1 + \norm{c_j}{1/2}{L^\infty(M')} + \norm{d (\mathrm{log}\, c_j)}{}{L^\infty(M')} \right) \norm{v_j}{}{H^1(M')}. \]
On the other hand,
\begin{align*}
\bigg| \Big\langle \Lambda_{g_j}(u_j) &\Big| (c_2^{-\frac{n - 2}{4}} - c_1^{-\frac{n - 2}{4}})v_k \Big\rangle \bigg| \\
& \leq \norm{\Lambda_{g_j}u_j}{}{H^{-1/2} (\partial M)} \norm{c_2^{-\frac{n - 2}{4}} - c_1^{-\frac{n - 2}{4}}}{}{C^1(\partial M)} \norm{v_k}{}{H^{1/2} (\partial M')} \\
& \leq \norm{u_j}{}{H^1 (M)} \norm{c_2^{-\frac{n - 2}{4}} - c_1^{-\frac{n - 2}{4}}}{}{C^1(\partial M)} \norm{v_k}{}{H^1 (M')}.
\end{align*}
Putting together the above estimates we prove the statement of the proposition.
\end{proof}

The estimate given in the previous proposition has terms that are not immediately controlled by $\norm{\Lambda_{g_1} - \Lambda_{g_2}}{}{\ast}$. However, these terms only depend on the difference of the conformal factors on the boundary. Since there is stability for this problem on the boundary we get the following corollary.
\begin{corollary} \label{cor:boundaryESc1c2} \sl Under the assumptions of Proposition \ref{prop:int_forCAL} we have that, for any $ v_j \in H^1(M') $ with $ j \in \{ 1, 2 \} $ solving
\[-\Delta_{g'} v_j + c^{-\frac{n - 2}{4}}_j \Delta_{g'} c^\frac{n - 2}{4}_j v_j = 0 \]
in $M'_\mathrm{int}$ the estimate
\begin{align*}
\bigg| \int_{M'} &(c_1^{-\frac{n- 2}{4}} \Delta_{g'} c_1^\frac{n- 2}{4} - c_2^{-\frac{n- 2}{4}} \Delta_{g'} c_2^\frac{n- 2}{4}) v_1 v_2 \, \dd V_{g'} \bigg| \\
\lesssim & C \norm{\Lambda_{g_1} - \Lambda_{g_2}}{\lambda}{\ast} \norm{v_1}{}{H^1(M')} \norm{v_2}{}{H^1(M')}
\end{align*}
is satisfied with $0 < \lambda <  2^{-2^{n+3}}$ and $C$ depending on $\normi{c_j}_{C^3(M)}$ and $\inf_M c_j$. The implicit constant in the inequality depends on $n$, $M$, $M'$, $ g' $ and $F$.
\end{corollary}

\begin{proof}
Fix a global coordinate system in $M$. We claim that one has 
\begin{equation*}
\normi{c_1-c_2}_{L^{\infty}(\partial M)} \leq C \normi{\Lambda_{g_1} - \Lambda_{g_2}}_*
\end{equation*}
and 
\begin{equation}
\normi{c_1 - c_2}_{C^1(\partial M)} + \normi{\partial_{\nu_g} c_1 - \partial_{\nu_g} c_2}_{L^{\infty}(\partial M)} \leq C \normi{\Lambda_{g_1} - \Lambda_{g_2}}_*^{\lambda} \label{es:stabilityBOUND}
\end{equation}
where the constant $C$ only depends on $n$, $\normi{c_j}_{C^3(M)}$,  $\normi{g}_{C^3(M)}$, $\inf_M c_j$, and the ellipticity constant $\inf_{x \in M} \inf_{v \in \R^n, |v|=1} g_{jk}(x) v^j v^k$ (these expressions involve the global coordinate system). Also, $\lambda = \lambda(n)$ is a number with $0 < \lambda(n) < 2^{-2^{n+3}}$. In fact, these two inequalites are an immediate consequence of the results of Kang and Yun \cite{KangYun}, see Theorem 1.3 and formula (4.12) in that paper.

From the second inequality above and from the a priori bounds for the coefficients, we obtain that 
\begin{multline*}
\normi{c_1^{-\frac{n - 2}{4}} - c_2^{-\frac{n - 2}{4}}}_{C^1(\partial M)} + \normi{\partial_{\nu_g} (c_1^{-\frac{n - 2}{4}}) - \partial_{\nu_g}(c_2^{-\frac{n - 2}{4}})}_{L^{\infty}(M)} \\
 \leq C \normi{\Lambda_{g_1} - \Lambda_{g_2}}_*^{\lambda}
\end{multline*}
for some constants $C$ and $\lambda$ as above. The result now follows from Proposition \ref{prop:int_forCAL}.
\end{proof}

\section{Complex geometrical optics solutions} \label{sec:CGOs}
In this section we review the properties of the CGOs constructed by Dos Santos Ferreira \textit{et al} in \cite{DSFKSaU} for admissible geometries. This construction has its roots in the paper \cite{KSU} by Kenig \textit{et al} in the context of the Calder\'on problem with partial data. However, we will follow a slight modification of the original argument given in \cite{KSaU}.

Throughout this section, $ M \subset \R \times M_0^\mathrm{int} $ will be an embedded $ n $-dimensional submanifold with boundary. The submanifold $ M $ will be assumed to be oriented and compact and it will be endowed with the Riemannian metric $ g = (e_\R \oplus g_0)|_M $. Thus, we are interested in constructing a family $ \{ u_\tau : \tau \geq \tau_0 \} \subset H^1(M) $ with $ \tau_0 \geq 1 $,
\begin{equation}
u_\tau = e^{-\tau(\varphi + i \psi)} (a + r_\tau) \label{id:CGOs}
\end{equation}
and such that
\begin{equation}
-\Delta_g u_\tau + q u_\tau = 0 \label{eq:schroedinger}
\end{equation}
in $ M^\mathrm{int} $ with $q \in L^\infty(M)$. Here $ \varphi $ and $ \psi $ are real-valued functions, $ a $ is a sort of complex amplitude and $ r_\tau $ is a correction term which becomes small when $ \tau $ increases.

Note that $u_\tau$ as in \eqref{id:CGOs} solves \eqref{eq:schroedinger} if and only if
\begin{align*}
e^{\tau(\varphi + i \psi)} (- \Delta_g +& \, q ) (e^{-\tau(\varphi + i \psi)} r_\tau) \\
=& \, (\Delta_g - q)a - \tau ( 2 \inner{d(\varphi + i\psi)}{da}_g + \Delta_g (\varphi + i\psi) a ) \\
&+ \tau^2 \inner{d(\varphi + i\psi)}{d(\varphi + i\psi)}_g a.
\end{align*}
The first idea in the construction of the CGOs is to arrange that the $\tau$ and $\tau^2$ terms on the right hand side of the previous identity vanish. Thus, for a suitable $\varphi$ we will look for $\psi$ and $a$ solving
\begin{equation}\label{eq:complex_eikonal}
|d \varphi|_g^2 = |d \psi|_g^2, \qquad  \langle d \varphi, d \psi \rangle_g = 0
\end{equation}
and
\begin{equation}\label{eq:transport}
2 \langle d( \varphi + i  \psi ), d a \rangle_g + \Delta_g (\varphi + i \psi) a = 0
\end{equation}
in $M$. The second idea is to provide a suitable $\varphi$ that allows us to solve the equation
\begin{equation}\label{eq:remainder}
e^{\tau(\varphi + i \psi)} (-\Delta_g + q) (e^{-\tau(\varphi + i \psi)} r_\tau) = (\Delta_g - q) a
\end{equation}
in $ M $. The appropriate candidates $\varphi$ to solve \eqref{eq:remainder} seem to be the limiting Carleman weights (LCWs for short), that were introduced in \cite{KSU} and characterized in \cite{DSFKSaU}. See \cite{Sa} for further discussion on LCWs.

At this point, we will choose $\varphi : \R \times M_0^\mathrm{int} \longrightarrow \R $ to be a LCW in $\R \times M_0^\mathrm{int}$. A natural choice is $ \varphi(s, \vartheta) = s $ for any $ (s, \vartheta) \in \R \times M_0^\mathrm{int} $. This choice makes the equations \eqref{eq:complex_eikonal} read as
\begin{equation*}
|d \psi|_g^2 = 1, \qquad  \partial_s \psi = 0.
\end{equation*}
The latter equation forces $ \psi : \R \times M_0^\mathrm{int} \longrightarrow \R $ to be independent of $ s $, and consequently, the former equation becomes a simple eikonal equation in $M_0^\mathrm{int}$. Since $M_0$ is simple, the function $\psi(s,\vartheta) = \dist_{g_0} (\omega, \vartheta) $ is a smooth solution of \eqref{eq:complex_eikonal} in $M_0^\mathrm{int}$  for any $\omega \in \partial M_0$. Here $ \dist_{g_0} (\omega, \centerdot) $ stands for the distance function from $ \omega $.

In order to solve \eqref{eq:transport}, we will choose local coordinates in $ \R \times M_0^\mathrm{int} $. Let $ \y_\omega : M_0^\mathrm{int} \longrightarrow \R^{n - 1} $ be Riemannian polar normal coordinates from the point $ \omega \in \partial M_0 $ with $\y_\omega(\vartheta) = (\rho, \theta_1, \dots, \theta_{n - 2})$ for any $\vartheta \in M_0^\mathrm{int}$. Since $\omega \in \partial M_0$ and $M_0$ is simple, one can choose $(\theta_1, \dots , \theta_{n - 2}) \in Q $ where $ Q = (0, \pi)^{n - 2} \subset \R^{n - 2}$. Define now $ \x_\omega : \R \times M_0^\mathrm{int} \longrightarrow \R^n $ as $\x_\omega(s, \vartheta) = (s, \y_\omega(\vartheta))$  for any $ (s, \vartheta) \in \R \times M_0^\mathrm{int} $. Note that in these coordinates $\psi (s,\rho,\theta_1, \dots, \theta_{n - 2}) = \rho$ and equation \eqref{eq:transport} becomes
\[ (\partial_s + i \partial_\rho) a + (\partial_s + i \partial_\rho) ( \log |g|^{1/4} ) a = 0 \]
in $\x_\omega (\R \times M_0^\mathrm{int})$. Multiplying by $ |g|^{1/4} $, we get the equation
\[ (\partial_s + i \partial_\rho) ( |g|^{1/4} a )  = 0. \]
Therefore we can choose $ a : \R \times M_0^\mathrm{int} \longrightarrow \C $ in such a way that in these coordinates $ a = |g|^{-1/4} \alpha \beta $ where $ \alpha = \alpha (s,\rho) $ satisfying $ (\partial_s + i \partial_\rho) \alpha = 0 $ in $ \R \times (0, R) $ with $R = \mathrm{diam}_{g_0}\, M_0$ and $\beta \in C^\infty_0 (Q)$.

We finally focus on equation \eqref{eq:remainder}. We write this equation in the following equivalent form
\begin{equation}\label{eq:remainder_1}
e^{\tau \varphi} (-\Delta_g + q) (e^{-\tau \varphi} \tilde{r}_\tau) = e^{- i \tau \psi} (\Delta_g - q) a,
\end{equation}
where $ \tilde{r}_\tau = e^{-i \tau \psi} r_\tau $. Let $ q \in L^\infty(\R \times M_0) $ still denote the extension by zero of $ q \in L^\infty(M) $. Let $ f $ denote the element in $ L^2(\R \times M_0) $ such that $ f_\tau = e^{- i \tau \psi} (\Delta_g - q) a $ almost everywhere in $ M $ and $ f_\tau = 0 $ almost everywhere else. By Theorem 4.1 in \cite{Sa} (see also Section 4 in \cite{KSaU}), we know that, for fixed $ \delta > 1/2 $, there exists a constant $C_0 \geq 1$ depending on $\delta, M, g_0$ such that, for all $ \tau \in \R $ with $ |\tau| \geq \max (1, C_0 \norm{q}{}{L^\infty(M)}) $ and $ \tau^2 $ out of the discrete set of the Dirichlet eigenvalues of $ -\Delta_{g_0} $, there exists a unique solution $ w_\tau \in H^1_{-\delta, 0} (\R \times M_0) $ of
\[ e^{\tau \varphi} (-\Delta_g + q) (e^{-\tau \varphi} w_\tau) = f_\tau \]
in $ \R \times M_0 $. Furthermore, this solution satisfies
\begin{equation*}
\begin{aligned}
\norm{w_\tau}{}{L^2_{-\delta}(\R \times M_0)} &\lesssim |\tau|^{- 1} \norm{(\Delta_g - q) a}{}{L^2(M)} \\
\norm{w_\tau}{}{H^1_{-\delta}(\R \times M_0)} &\lesssim \norm{(\Delta_g - q) a}{}{L^2(M)}.
\end{aligned}
\end{equation*}
For the sake of completeness, let us provide the definitions of the spaces introduced above
\begin{align*}
L^2_{-\delta} (\R \times M_0) &= \{ u \in L^2_\mathrm{loc} (\R \times M_0) : (1 + s^2)^{-\delta/2} u \in L^2(\R \times M_0) \}, \\
H^1_{-\delta} (\R \times M_0) &= \{ u \in L^2_{-\delta} (\R \times M_0) : |du| \in L^2_{-\delta} (\R \times M_0) \}, \\
H^1_{-\delta, 0} (\R \times M_0) &= \{ u \in H^1_{-\delta} (\R \times M_0) : u|_{\R \times \partial M_0} = 0 \};
\end{align*}
and their corresponding norms
\begin{align*}
\| u \|_{L^2_{-\delta} (\R \times M_0)} &= \| (1 + s^2)^{-\delta / 2} u \|_{L^2 (\R \times M_0)}, \\
\| u \|_{H^1_{-\delta} (\R \times M_0)} &= \| u \|_{L^2_{-\delta} (\R \times M_0)} + \| |du| \|_{L^2_{-\delta} (\R \times M_0)}.
\end{align*}

Finally, we end the construction of CGOs taking $ \tilde{r}_\tau = w_\tau|_{M^\mathrm{int}} $. The implicit constants only depend on $\delta, M$ and $g_0$.

We end this section by stating more succinctly the existence of the CGOs.

\begin{proposition} \label{prop:CGOs} \sl There exists a constant $C_0 \geq 1$ depending on $M$ and $g_0$ such that for 
\[|\tau| \geq \max(C_0 \norm{q}{}{L^\infty(M)}, 1) \qquad \tau^2 \notin \spec(-\Delta_{g_0}), \]
the function
\[u_\tau = e^{-\tau(\varphi + i \psi)} (a + r_\tau),\]
with $\varphi(s, \vartheta) = s$, $\psi(s,\vartheta) = \dist_{g_0} (\omega, \vartheta)$ and $a = |g|^{-1/4} \alpha \beta$ where $\alpha$ solves $(\partial_s + i \partial_\rho)\alpha = 0$ and $\beta \in C^\infty_0 (Q)$, is a solution of
\[ -\Delta_{g} u_\tau + q u_\tau = 0\]
in $M^{\inte}$. Moreover,
\begin{equation}
\norm{e^{-{i\tau \psi}} r_\tau}{}{H^k(M)} \lesssim |\tau|^{-1+k} \norm{(\Delta_g - q)a}{}{L^2(M)}
\label{es:L2remainder} 
\end{equation}
for $k = 0, 1$. The implicit constant only depends on $M$ and $g_0$.
\end{proposition}

\section{Stability estimates} \label{sec:stability}
In this section we will provide the stability estimates for the problems under consideration, namely, controlling either the difference of the Schr\"odinger potentials or the difference of the conformal factors by their corresponding boundary data. The basic idea will be to plug the CGOs from Section \ref{sec:CGOs} into the inequalities given either in Proposition \ref{prop:boundaryESq1q2} or Corollary \ref{cor:boundaryESc1c2}.

Since the arguments to show the estimates announced for the two considered IBVP are quite similar, we will do both at the same time. Thus, if we are considering the IBVP associated to the Schr\"odinger operator, we agree the following notation:
\[ q = c (q_1 - q_2), \qquad \varepsilon = \dist (C_{q_1}, C_{q_2}), \]
$v_j$ is one of the solutions for
\[-\Delta_{g'} v_j + ( c^{-\frac{n - 2}{4}} \Delta_{g'} c^\frac{n - 2}{4}  + c q_j) v_j = 0 \]
constructed in Section \ref{sec:CGOs} and the implicit constants only depend on $n, M, M', F, g', c$ and $K$. Here $K$ is as in Theorem \ref{th:sch}. However, if we consider the simplification of the generalized Calder\'on problem, then 
\[ q = c_1^{-\frac{n- 2}{4}} \Delta_{g'} c_1^\frac{n- 2}{4} - c_2^{-\frac{n- 2}{4}} \Delta_{g'} c_2^\frac{n- 2}{4}, \qquad \varepsilon = \norm{\Lambda_{g_1} - \Lambda_{g_2}}{\lambda}{\ast} \]
where $\lambda$ is as in Corollary \ref{cor:boundaryESc1c2}, $v_j$ is one of the solutions for
\[-\Delta_{g'} v_j + c^{-\frac{n - 2}{4}}_j \Delta_{g'} c^\frac{n - 2}{4}_j v_j = 0 \]
constructed in Section \ref{sec:CGOs} and the implicit constants only depend on $n, M, M', F, g'$ and $K$. Here $K$ is as in Theorem \ref{th:cal}.

We now start with the argument. Let $\omega $ belong to $ \partial M_0$ and consider
\[ v_1 = e^{\tau (\varphi + i \psi)} (a_1 + r_1), \qquad v_2 = e^{-\tau(\varphi + i \psi)} (a_2 + r_2) \] constructed as in Section \ref{sec:CGOs}, where we choose $a_1 = \alpha \beta |g'|^{-1/4}$ and $a_2 = |g'|^{-1/4}$ in the coordinates used in that section. Then, either Proposition \ref{prop:boundaryESq1q2} or Corollary \ref{cor:boundaryESc1c2} implies
\begin{align*}
\bigg| \int_{M'} q a_1 a_2 \, \dd V_{g'} \bigg| \lesssim & \varepsilon \norm{v_1}{}{H^1(M')} \norm{v_2}{}{H^1(M')} + \norm{a_1}{}{L^2(M')} \norm{r_2}{}{L^2(M')} \\
& + \norm{r_1}{}{L^2(M')} + \norm{r_1}{}{L^2(M')} \norm{r_2}{}{L^2(M')}.
\end{align*}
Recall from Section \ref{sec:CGOs} that $ Q = (0, \pi)^{n - 2} \subset \R^{n - 2}$ and $ R = \mathrm{diam}_{g_0} M_0$. Moreover, introduce some other notation:
\[ S = \max \{ |s| : \exists \vartheta \in M_0,\, (s, \vartheta) \in M'  \} \qquad Q' = (- S, S) \times (0, R). \]
Now using the form of the solution $v_1$ and $v_2$ and estimates labelled with \eqref{es:L2remainder}, we get
\begin{equation}
\bigg| \int_{M'} q a_1 a_2 \, \dd V_{g'} \bigg| \lesssim (\varepsilon e^{k \tau} + \tau^{-1}) \norm{\alpha}{}{H^2(Q')} \norm{\beta}{}{H^2(Q)}
\label{es:firstsES}
\end{equation}
where $k > 2 (S + R)$, the implicit constant depends also on $R$ and $\tau \geq C_0 K $ with $\tau^2$ out of the discrete set of Dirichlet eigenvalues of $-\Delta_{g_0}$ and $C_0$ as in Proposition \ref{prop:CGOs}.

In order to extract information from the left hand side of \eqref{es:firstsES}, we choose $\alpha (s, \rho) = e^{-\sigma (\rho + is)}$ with $\sigma \in \R$ and check that it becomes
\begin{equation}
\bigg| \int_Q \beta(\theta) \int_0^R \hat{q}(\sigma, \rho, \theta) e^{-\sigma \rho} \, \dd \rho \, \dd \theta \bigg| \lesssim (\varepsilon e^{k \tau} + \tau^{-1}) e^{k |\sigma|} \norm{\beta}{}{H^2(Q)},
\label{es:LHSfirstsES}
\end{equation}
where $\hat{q}(\centerdot, \rho, \theta)$ denotes the Fourier transform of (the zero extension of) $q(\centerdot, \y_\omega^{-1}(\rho, \theta))$ in the $s$ variable and $\dd \theta$ is the euclidean volume form in $Q$. Note that the integrand of $\dd \theta$ on the left hand side of \eqref{es:LHSfirstsES} means, at the level of the manifold $M_0$, integrating the Fourier transform of $q$ along a geodesic (starting from $\omega$ with direction described by $\theta$) with respect to the weight $e^{-\sigma \rho}$. This brings naturally to this context the attenuated geodesic ray transform (see for instance \cite{DSFKSaU}, \cite{SaU}).

In order to define the attenuated geodesic ray transform, let us introduce some notation. The unit sphere bundle on $M_0$ is denoted by $SM_0$ and defined by
\[ SM_0 = \bigcup_{\vartheta \in M_0} S_\vartheta,\qquad S_\vartheta = \{ (\vartheta, X_\vartheta) : X_\vartheta \in T_\vartheta M_0 : |X_\vartheta|_{g_0} = 1 \}. \]
For notational convenience, we drop the subindex referring the point and we write $X$ instead of $X_\vartheta$. This manifold $SM_0$ has as boundary $\partial SM_0 = \{ (\vartheta, X) \in SM_0: \vartheta \in \partial M_0 \}$. Let $ N_0 $ denote the unit vector field on $ \partial M_0 $ pointing outward and define the manifold
\[ \partial_+ SM_0 = \{ (\vartheta, X) \in \partial SM_0 : g_0(X, N_0) \leq 0 \} \]
whose boundary is given by $ \{ (\vartheta, X) \in \partial SM_0 : g_0(X, N_0) = 0 \} $. Thus the space $C^\infty_0 ((\partial_+ SM_0)^\mathrm{int})$ denote the smooth functions on $\partial_+ SM_0$ vanishing near tangential directions.

Let $t \mapsto \gamma(t; \vartheta, X)$ denote the unit speed geodesic starting at $\vartheta \in M_0$ in direction $X$ and let $\T(\vartheta, X)$ be the time when geodesic exits $M_0$. Since $(M_0, g_0)$ is simple, $\T(\vartheta, X)$ is finite for every $(\vartheta, X) \in SM_0$. Let the geodesic flow be denoted by $\phi_t(\vartheta, X) = (\gamma(t; \vartheta, X), \dot{\gamma}(t; \vartheta, X))$, where $\dot{\gamma}(t; \vartheta, X)$ denote the tangent vector at $\gamma(t; \vartheta, X)$. Thus, the attenuated geodesic ray transform,  with attenuation $ - \sigma $, of a continuous function $ f $ defined on $ M_0 $ is defined by
\[ I_\sigma f (\vartheta, X) = \int_0^{\T(\vartheta, X)} f(\gamma(t; \vartheta, X)) e^{-\sigma t} \dd t, \qquad \forall (\vartheta, X) \in \partial_+ SM_0. \]
Before going further, let us introduce another operator. Let $h $ belong to $ C^\infty_0 ((\partial_+ SM_0)^\mathrm{int})$, define
\[ I^*_\sigma h (\vartheta) = \int_{S_\vartheta} e^{- \sigma \T(\vartheta, -X)} h (\phi_{-\T(\vartheta, -X)}(\vartheta, X)) \dd S_\vartheta (X) \]
where $\dd S_\vartheta$ denotes the natural Riemannian volume form on $S_\vartheta$.

For the point $\omega \in \partial M_0$ considered above and some $\delta > 0$, we take coordinates
\[ \Theta_\omega : S_\omega^\delta = \{ X \in S_\omega : g_0(X, N_0) < - \delta \} \longrightarrow Q \]
such that, given $b \in C^\infty_0 ((\partial_+ SM_0)^\mathrm{int})$ with $ \supp b(\omega, \centerdot) \subset \overline{S_\omega^\delta}$, $\beta$ can be chosen to satisfy 
\[ b(\omega, \centerdot) \dd S_\omega = \Theta^\ast (\beta \dd \theta), \]
where $\Theta^\ast$ the pull-back of $\Theta$. Thus we see that
\begin{equation}
\begin{aligned}
\int_Q &\beta(\theta) \int_0^R \hat{q}(\sigma, \rho, \theta) e^{-\sigma \rho} \, \dd \rho \, \dd \theta \\
&= \int_{S_\omega^\delta} b (\omega, X) \left( \int^{\T(\omega, X)}_0 \hat{q}(\sigma, \gamma(r; \omega, X)) e^{-\sigma r} \, \dd r \right) \, \dd S_\omega (X)
\end{aligned}
\label{eq:qa1a2RAY}
\end{equation}
where we agreed to denote $\mathcal{F} [q (\centerdot, \gamma(r; \omega, X))](\sigma)$ (the Fourier transform with respect to the $s$ variable) by $\hat{q}(\sigma, \gamma(r; \omega, X))$. Observe that $q$ is not good enough to give pointwise meaning to $I_\sigma (\hat{q}(\sigma, \centerdot))$, however, Fubini's theorem ensures that this is in $L^1 (\partial_+ SM_0)$. Thus, integrating \eqref{eq:qa1a2RAY} over $\partial M_0$ and using \eqref{es:LHSfirstsES} we can get
\begin{equation}
\begin{aligned}
\bigg| \int_{\partial_+ SM_0} &b (\omega, X) I_\sigma (\hat{q}(\sigma, \centerdot))(\omega, X) \, \dd (\partial SM_0) \bigg| \\
&\lesssim (\varepsilon e^{k \tau} + \tau^{-1}) e^{k |\sigma|} \int_{\partial M_0} \norm{b(\omega, \centerdot)}{}{H^2(S_\omega^\delta)}\, \dd A_{g_0},
\end{aligned}
\label{es:casiLLEGAMOS}
\end{equation}
where the implicit constant depends on $\delta$. Here $ \dd (\partial SM_0) $ denotes the natural Riemannian volume form on $\partial SM_0$ and $\dd A_{g_0}$ is the surface element on $\partial M_0$.

We next choose $b (\omega, X)$ to be $ \mu(\omega, X) I_\sigma f (\omega, X)$ with $ \mu(\omega,X) = - g_0(X,N_0)$ for $f \in C^\infty_0(M_1^\mathrm{int})$ with $M_1 $ a compact subset of $ M_0^\mathrm{int}$ to be chosen later. With this choice, we would like to show that
\begin{equation}
\int_{\partial_+ SM_0} I_\sigma f\, I_\sigma (\hat{q}(\sigma, \centerdot)) \mu \, \dd (\partial SM_0) = \int_{M_0} f \, I_\sigma^\ast I_\sigma (\hat{q}(\sigma, \centerdot)) \, \dd V_{g_0}.
\label{es:estamosLLEGANDO}
\end{equation}
From Lemma 5.4 in \cite{DSFKSa}, we know that
\begin{equation}
\int_{\partial_+ SM_0} I_\sigma f\, h\, \mu \, \dd (\partial SM_0) = \int_{M_0} f \, I_\sigma^\ast h \, \dd V_{g_0}
\label{id:adjuntaI}
\end{equation}
whenever $f \in C^\infty (M_0)$ and $h \in C^\infty_0 ((\partial_+ SM_0)^\mathrm{int})$. However, this is not enough for us since $I_\sigma (\hat{q}(\sigma, \centerdot))$ only belongs to $L^1(\partial_+ SM_0)$. Fortunately, this still holds for $h \in L^1(\partial_+ SM_0)$.
\begin{lemma} \sl Identity \eqref{id:adjuntaI} holds for $f \in C^\infty (M_0)$ and $h \in L^1(\partial_+ SM_0)$. Consequently, \eqref{es:estamosLLEGANDO} also holds.
\end{lemma}
\begin{proof} It will be convenient to introduce the following notation
\[h_\psi (y, \eta) = h (\phi_{-\T(y, -\eta)}(y, \eta)) \]
for all $(y, \eta) \in SM_0$. Note that $h_\psi (\phi_t (x, \xi)) = h (x, \xi)$ for all $(x, \xi) \in \partial_+ SM_0$. Hence
\[\int_{\partial_+ SM_0} I_\sigma f\, h\, \mu \, \dd (\partial SM_0) = \int_{\partial_+ SM_0} J\, \mu \, \dd (\partial SM_0) \]
with
\[J(x, \xi) = \int_0^{\T(x,\xi)} f(\gamma(t; x, \xi)) e^{- \sigma \T(\gamma(t;x,\xi), - \dot{\gamma}(t;x,\xi))} h_\psi(\phi_t(x,\xi)) \, \dd t. \]
It was proven in Lemma 3.3.2 from \cite{Sh} (see also Lemma A.8 in \cite{DaPaStU}), that the pull-back of $\dd SM_0$ through the diffeomorphism $(t;x,\xi) \in D \longrightarrow \phi_t (x, \xi) \in SM_0 \setminus T\partial M_0 $ with
\[ D = \{ (t;x,\xi) : (x, \xi) \in \partial_+ SM_0, \, t \in [0, \T(x,\xi)] \} \]
is given by $\mu \, \dd( \partial_+ SM_0 )\wedge \dd t$. Therefore, $h_\psi \in L^1(SM_0)$ since $h_\psi$ is constantly equal to $h(x, \xi)$ through $\{ \phi_t(x, \xi) : t \in [0, \T(x,\xi) \}$ and $h \in L^1(\partial_+ SM_0)$, and
\begin{align*}
\int_{\partial_+ SM_0} J\, \mu \, \dd (\partial SM_0)& = \int_{SM_0} f(y) e^{-\sigma \T(y, -\eta)} h_\psi(y, \eta) \, \dd SM_0(y, \eta)\\
& = \int_{M_0} f \, I_\sigma^\ast h \, \dd V_{g_0}
\end{align*}
by using Fubini's theorem twice. This proves that \eqref{id:adjuntaI} holds for $h \in L^1(\partial_+ SM_0)$. Identity \eqref{es:estamosLLEGANDO} is then an immediate consequence.
\end{proof}

Finally a straightforward computation in normal coordinates based at $\omega$ gives
\[ \bigg| \int_{M_0} f \, I_\sigma^\ast I_\sigma (\hat{q}(\sigma, \centerdot)) \, \dd V_{g_0} \bigg| \lesssim (\varepsilon e^{k \tau} + \tau^{-1}) e^{k |\sigma|} \norm{f}{}{H^2(M_0)} \]
for $k > 2 (S + R)$, $f \in C^\infty_0(M_1^\mathrm{int})$ and $\tau \geq C_0 K $ with $\tau^2$ out of the discrete set of Dirichlet eigenvalues of $-\Delta_{g_0}$ and $C_0$ as in Proposition \ref{prop:CGOs}. Next we will make a choice for $\tau$ in terms of $\varepsilon$. Firstly note that $k$ can be chosen larger if necessary to avoid that $( |\log \varepsilon|/(2k) )^2$ is in the set of Dirichlet eigenvalues of $-\Delta_{g_0}$. Moreover, if $\varepsilon \leq e^{-2 k C_0 K}$ we can take
\[ \tau = \frac{1}{2k} |\log \varepsilon| \]
to obtain
\begin{equation}
\norm{I_\sigma^\ast I_\sigma (\hat{q}(\sigma, \centerdot))}{}{H^{-2}(M_1)} \lesssim (\varepsilon^{1/2} + |\log \varepsilon|^{-1}) e^{k |\sigma|}.
\label{es:stabNORMAL}
\end{equation}

The idea now will be to use the ellipticity of the normal operator $I^\ast_\sigma I_\sigma$ to obtain an estimate for $\hat{q}(\sigma, \centerdot)$. To do so, choose $M_1 \subset M_0^\mathrm{int}$ to satisfy the following assertion: there exist $M_2$ and $M_3$ two compact subsets of $M_1^\mathrm{int}$ such that
\[ M' \subset (-S, S) \times M_3^\mathrm{int},\qquad M_3 \subset M_2^\mathrm{int}. \]
Note that $\supp \hat{q}(\sigma, \centerdot) \subset M_3$ for all $\sigma \in \R$.
\begin{lemma} \label{lem:normalOPstability} \sl
Let $M_1, M_2$ and $M_3$ as above. Then there exists a $\delta_0 > 0$ such that
$$
\left\lVert f \right\rVert_{H^{-k}(M_1)} \lesssim \max_{|\sigma| \leq \delta_0} \left\lVert I_{\sigma}^* I_{\sigma} f \right\rVert_{H^{-k+1}(M_1)}, \quad \forall f \in H^{-k}_{M_3}(M_1).
$$
The implicit constant here depends on $\delta_0$.
\end{lemma}
Recall that $H^{-k}_{M_3}(M_1)$ is the space of elements of $H^{-k}(M_1)$ whose support is contained in $M_3$.
\begin{proof}
Write $N = I_{\sigma}^* I_{\sigma}$. By Proposition 2 in \cite{FStU}, we know that $N$ is an elliptic pseudodifferential operator of order $-1$ in $M^{\text{int}}_1$, and there is a pseudodifferential operator $Q$ of order $1$ in $M^{\text{int}}_1$ and an operator $R$ with kernel in $C^{\infty}_0(M^{\text{int}_1} \times M^{\text{int}}_1)$ such that, if $f \in H^{-k}_{M_3}(M_1)$, then
$$
\chi QNf = f + \chi Rf
$$
with $\chi \in C^{\infty}_0(M^{\text{int}}_2)$ with $\chi = 1$ near $M_3$. Therefore
$$
\left\lVert f \right\rVert_{H^{-k}(M_3)} \lesssim \left\lVert Nf \right\rVert_{H^{-k+1}(M_2)} + \left\lVert f \right\rVert_{H^\lambda (M_2)},
$$
for any fixed real $\lambda$. Let $X = H^{-k}_{M_3}(M_1)$, $Y = C ( [-\delta_0, \delta_0]; H^{-k+1}(M_2))$, and $Z = H^\lambda(M_2)$. The operator $N$ is bounded from $X$ to $Y$ since $N$ is of order $-1$ and $M_2 \subset M_1^\mathrm{int}$, and the injection from $X$ to $Z$ is compact if $\lambda$ is small enough. Also, $N$ is injective, since for any $f \in X$ with $Nf = 0$ one has $f \in C^{\infty}_0(M^{\text{int}}_1)$ by elliptic regularity and
$$
\left\lVert I_{\sigma} f \right\rVert_{L^2_{\mu}(\partial_+(SM_1))}^2 = (Nf, f)_{L^2(M_1)} = 0
$$
and $I_{\sigma} f = 0$. By \cite{DSFKSaU}, we know that there exists a $\delta_0 > 0$ such that if $|\sigma| \leq \delta_0$ then $f = 0$. By using Lemma 2 in \cite{StU},  we have 
$$
\left\lVert f \right\rVert_{H^{-k}(M_1)} \lesssim \max_{|\sigma| \leq \delta_0} \left\lVert Nf \right\rVert_{H^{-k+1}(M_2)}.
$$
This implies the result.
\end{proof}

Therefore we know the following estimate
\begin{equation}
\norm{\hat{q}(\sigma, \centerdot)}{}{H^{-3}(M_1)} \lesssim \varepsilon^{1/2} + |\log \varepsilon|^{-1} \, \qquad \forall |\sigma| \leq \delta_0.
\label{es:Fourier_q}
\end{equation}
Let us remark that the the implicit constant here also depends on $\delta_0$ and the estimate holds when $\varepsilon \leq e^{-2 k C_0 K}$.

The next step will be to control a mixed norm for $q$. Since the range of $\sigma$ for which \eqref{es:Fourier_q} holds can be very small, we will need to make use of the analytic properties of $\hat{q}(\sigma, \centerdot)$ (recall that $q$ was compactly supported) to control $\hat{q}(\sigma + i, \centerdot)$ for $|\sigma| \leq \tilde R$ with $\tilde R$ arbitrarily large. This will be enough to bound a mixed norm for $q$ by the boundary data. In \cite{HW} Heck and Wang used a result by Vessella \cite{V} to control an arbitrary large set of low frequencies by a small one. Our approach here is slightly different and is based on properties of subharmonic functions. The argument is due to Dos Santos Ferreira and has been used in \cite{CaDSFR} to deal with a similar situation.

The first step will be to obtain from \eqref{es:Fourier_q} an estimate for certain subharmonic function. Note that \eqref{es:Fourier_q} implies
\begin{equation}
| \langle \hat{q}(\sigma, \centerdot), f \rangle | \lesssim (\varepsilon^{1/2} + |\log \varepsilon|^{-1}) \| f \|_{H^3(M_1)} \label{es:sh1}
\end{equation}
for all $|\sigma| \leq \delta_0$ and $f \in C^\infty_0(M_1)$. On the other hand, since $q$ is compactly supported in $[-S, S] \times M_1$, the analytic extension of the Fourier transform of $q$ in the Euclidean direction satisfies
\begin{equation}
| \langle \hat{q}(\sigma + i \eta, \centerdot), f \rangle | \lesssim e^{S \eta} \| f \|_{H^3(M_1)} \label{es:sh2}
\end{equation}
for all $\sigma + i \eta \in \C$ such that $\eta \geq 0$. Define
\[F(\sigma, \eta) = \log \frac{| \langle \hat{q}(\sigma + i \eta, \centerdot), f \rangle |}{C \| f \|_{H^3(M_1)}} - S\eta, \qquad (\sigma, \mu) \in \R^2,\]
where $C$ is the sum of the implicit constants in \eqref{es:sh1} and \eqref{es:sh2}. Note that $F$ is subharmonic and satisfies
\begin{align*}
F(\sigma, \eta) &\leq \log (\varepsilon^{1/2} + |\log \varepsilon|^{-1}) & & 0 \leq \sigma \leq \delta_0, \\
F(\sigma, \eta) &\leq 0 & & \sigma \in \R,\, \eta \geq 0.
\end{align*}
Next, we will show a lemma that allows to transmit the smallness of $F$ in the segment $\{ (\sigma, 0) : |\sigma| \leq \delta_0 \}$ to $\{ (\sigma, 1) : |\sigma| \leq \tilde{R} \}$ where $\tilde{R}$ is arbitrarily large.

\begin{lemma} \label{lem:subharmonic} \sl
Let $b$ and $\delta$ be positive constants and let $F$ be a subharmonic function in an open neighbourhood of
\[\{ (x, y) \in \R^2 : y \geq 0 \}\]
such that
\begin{align*}
F(x, 0) &\leq -b & & 0 \leq x \leq \delta, \\
F(x, y) &\leq 0 & & x\in \R,\,  y \geq 0.
\end{align*}
Then
\[F(x,y) \leq -\frac{b}{\pi} \left( \arctan \frac{x+ \delta}{ y} - \arctan\frac{x-\delta}{y} \right)\]
for all $(x, y) \in \R^2$ such that $y \geq 0$.
\end{lemma}
\begin{proof}
Consider the Poisson kernel for $\{ (x, y) \in \R^2 : y > 0 \}$
\[P_y(x) = \frac{1}{\pi} \frac{y}{x^2 + y^2}, \qquad x \in \R,\, y \geq 0.\]
Then,
\[u (x, y) = \frac{1}{\pi} \int_{-\delta}^\delta \frac{y}{(x - z)^2 + y^2} \, \dd z\]
is harmonic in $\{ (x, y) \in \R^2 : y > 0 \}$ and $u(x, 0) = 1$ for all $|x| \leq \delta$ and $u(x, 0) = 0$ for $|x| > \delta$. Thus,
\[F(x, 0) \leq -b u(x, 0), \qquad \forall x \in \R.\]
Moreover, for every $\varepsilon > 0$ there exists $R > 0$ such that
\[ -b u(x, y) + \varepsilon \geq 0, \qquad |x| + |y| = R,\]
since $u(x, y) \longrightarrow 0$ as $|x| + |y| \longrightarrow \infty$. Therefore,
\[F(x, y) \leq -b u(x,y) + \varepsilon\]
on $\{ (x, 0) : x \in \R \} \cup \{ (x, y) : |x| + |y| = R,\, y\geq 0 \}$. By the properties of subharmonic functions
\[F(x, y) \leq -b u(x,y) + \varepsilon\]
in $\{ (x, y) : |x| + |y| \leq R,\, y\geq 0 \}$. Making $\varepsilon$ vanish and computing $u(x, y)$ explicitly, we deduce the statement of the lemma.
\end{proof}
Whenever $\varepsilon^{1/2} + |\log \varepsilon|^{-1} < 1$, Lemma \ref{lem:subharmonic} can be applied and yields
\[| \langle \hat{q}(\sigma + i, \centerdot), f \rangle | \lesssim \| f \|_{H^3(M_1)} e^{\tilde{k}\log(\varepsilon^{1/2} + |\log \varepsilon|^{-1})/\tilde{R}^2}\qquad \forall |\sigma| \leq \tilde{R}
\]
with $\tilde{R} \geq 1$, since $\arctan (x+ \delta_0) - \arctan(x-\delta_0) \sim 2\delta_0/x^2$ for $x \geq 1$. This provides a control of the frequencies $\sigma + i$ with $|\sigma| \leq \tilde{R}$:
\[\| \hat{q}(\sigma + i, \centerdot) \| \lesssim e^{\tilde{k}\log(\varepsilon^{1/2} + |\log \varepsilon|^{-1})/\tilde{R}^2}.\]
The control of these frequencies and the fact that $q \in H^\lambda (\R; H^{-3}(M_1))$ will be enough to bound a mixed norm for $q$. The choice of $\lambda < 1/2$ guarantees that the extension by zero preserve the regularity and $\norm{q}{}{H^\lambda(\R;H^{-3}(M_1))}$ is bounded by a constant depending on the a priori bound $K$, $n$ and $M$. Indeed,
\begin{align*}
\| q \|^2_{L^2(\R; H^{-3}(M_1))} \leq& e^{2S} \int_{\R} \| \hat{q}(\sigma + i, \centerdot) \|^2_{H^{-3}(M_1)} \dd \sigma \\
\lesssim& \tilde{R} e^{\tilde{k}\log(\varepsilon^{1/2} + |\log \varepsilon|^{-1})/\tilde{R}^2}\\
&+ \tilde{R}^{-2\lambda} \int_{|\sigma| > \tilde{R}} (1 + |\sigma|^2)^\lambda \| \tilde{q}(\sigma + i, \centerdot) \|^2_{H^{-3}(M_1)} \,\dd \sigma.
\end{align*}

Finally, choosing
\[\frac{\tilde{k}}{\tilde{R}^2} = |\log(\varepsilon^{1/2} + |\log \varepsilon|^{-1})|^{-1/2}\]
we get
\begin{equation}
\norm{q}{}{L^2(\R; H^{-3} (M_1))} \lesssim \bigg| \log (\varepsilon^{1/2} + |\log \varepsilon|^{-1}) \bigg|^{-\lambda/4} \label{es:stabilityq}
\end{equation}
whenever $\varepsilon$ is small enough. The implicit constant in the last estimate depends also on $\lambda$. This ends the proof of Theorem \ref{th:sch}. However, we will need an extra argument to prove Theorem \ref{th:cal}, which is as follows. Note that in this case
\[c_1^\frac{n- 2}{4} c_2^\frac{n- 2}{4} q = |g'|^{-1/2} \partial_{x_j} \left( c_1^\frac{n- 2}{4} c_2^\frac{n- 2}{4} g'^{j k} |g'| \partial_{x_k} (\log c_1^\frac{n- 2}{4} - \log c_2^\frac{n- 2}{4}) \right).\]
Here we are using Einstein's summation convention. Observe that $\log c_1 - \log c_2$ satisfies an elliptic equation so, by its well-posedness, we have
\[ \norm{\log c_1 - \log c_2}{}{H^1(M)} \lesssim \norm{q}{}{H^{-1} (M)} + \norm{\log c_1 - \log c_2}{}{H^{1/2}(\partial M)}. \]
By a simple interpolation argument, the a priori bounds for $c_1$ and $c_2$ and estimate \eqref{es:stabilityq} we get that
\begin{align*}
\norm{q}{}{H^{-1} (M)} &\leq \norm{q}{}{L^2(\R;H^{-1}(M_1))} \leq \norm{q}{1/3}{L^2(\R;H^{-3}(M_1))} \norm{q}{2/3}{L^2(M)}\\
& \lesssim \bigg| \log (\varepsilon^{1/2} + |\log \varepsilon|^{-1}) \bigg|^{-\lambda/12}.
\end{align*}
Therefore, using \eqref{es:stabilityBOUND} we get
\[ \norm{\log c_1 - \log c_2}{}{H^1(M)} \lesssim \bigg| \log (\varepsilon^{1/2} + |\log \varepsilon|^{-1}) \bigg|^{-\lambda/12} + \varepsilon. \]
Finally, by interpolation and Morrey's embedding (in the spirit of \cite{CaGR}) we conclude the proof of Theorem \ref{th:cal}.

\begin{acknowledgements} \rm P.C. is supported by the projects ERC-2010 Advanced Grant, 267700 - InvProb and Academy of Finland (Decision number 250215, the Centre of Excellence in Inverse Problems). P.C. also belongs to the project MTM 2011-02568 Ministerio de Ciencia y Tecnolog\'ia de Espa\~na. M.S. is partly supported by the Academy of Finland and an ERC Starting Grant. The authors would like to thank the organizers of the program on Inverse Problems held in the Institut Mittag-Leffler in 2013 where part of this work was carried out. They would also like to thank David Dos Santos Ferreira for sharing his elegant argument to enlarge the set of controlled frequencies.
\end{acknowledgements}

\end{document}